\UseRawInputEncoding
\pdfoutput=1
\documentclass[journal]{IEEEtran}
\usepackage[english]{babel}

\usepackage{algorithm,algorithmic}

\RequirePackage{algorithm}

\makeatletter
\newenvironment{breakablealgorithm}
  {
    \begin{center}
      \refstepcounter{algorithm}
      \hrule height.8pt depth0pt \kern2pt
      \parskip 0pt
      \renewcommand{\caption}[2][\relax]{
        {\raggedright\textbf{\fname@algorithm~\thealgorithm} ##2\par}%
        \ifx\relax##1\relax 
          \addcontentsline{loa}{algorithm}{\protect\numberline{\thealgorithm}##2}%
        \else 
          \addcontentsline{loa}{algorithm}{\protect\numberline{\thealgorithm}##1}%
        \fi
        \kern2pt\hrule\kern2pt
     }
  }
  {
     \kern2pt\hrule\relax
   \end{center}
  }
\makeatother

\usepackage{caption}

\usepackage{graphicx}
\usepackage{float} 
\usepackage{subcaption}
\usepackage{multirow}

\usepackage{subcaption}

\usepackage{amssymb}
\usepackage{amsmath}
\usepackage{amsthm}

\usepackage{cite}
\newtheorem{theorem}{Theorem}

\newtheorem{lemma}{Lemma}

\newtheorem{remark}{Remark}

\usepackage{graphicx}
\usepackage[colorlinks=true, allcolors=blue]{hyperref}

\usepackage{cleveref}
\crefname{figure}{Fig.}{Figs.}

\begin{document}

\title{Online Learning-Based Predictive Control for Nonlinear System}

\author{}
\author{Yuanqing Zhang, Huanshui Zhang

\thanks{Y. Zhang is with the School of Control Science and Engineering, Shandong University, Jinan, Shandong, China, 250061 (e-mail: 202320734@mail.sdu.edu.cn). }

\thanks{H. Zhang is with the College of Electrical Engineering and Automation, Shandong University of Science and Technology, Qingdao, Shandong, China, 266590 (e-mail: hszhang@sdu.edu.cn).}}

\markboth{Journal of \LaTeX\ Class Files,~Vol.~14, No.~8, August~2015}%
{Shell \MakeLowercase{\textit{et al.}}: Bare Demo of IEEEtran.cls for IEEE Journals}

\maketitle
\begin{abstract}
In this paper, we propose an online learning-based predictive control (LPC) approach designed for nonlinear systems that lack explicit system dynamics. Unlike traditional model predictive control (MPC) algorithms that rely on known system models to optimize controller outputs, our proposed algorithm integrates a reinforcement learning component to learn optimal policies in real time from the offline dataset and real-time data. Additionally, an optimal control problem (OCP)-based optimization framework is incorporated to enhance real-time computational efficiency while ensuring stability during online operation. Moreover, we rigorously establish the super-linear convergence properties of the algorithm. Finally, extensive simulations are performed to evaluate the feasibility and effectiveness of the proposed approach.
\end{abstract}


\IEEEpeerreviewmaketitle

\section{Introduction}

Model predictive control (MPC), also known as receding horizon control, is a powerful tool designed to solve the problem of optimal control in a receding horizon. Because of its excellent performance, it is widely used in the fields of vehicles \cite{ref1,ref2,ref3}, power electronics \cite{ref4,ref5} and industrial process control \cite{ref6}. Consequently, the study of MPC remains of great practical significance.
 
Existing MPC methods are generally model-based, relying on predefined system dynamics. Traditional MPC solves an optimization problem in each predictive horizon, generating an optimal predictive control sequence, executing only the first control action in the sequence, then updates the state and repeats the process. This control method has enabled optimal control of both linear and nonlinear systems \cite{ref7,ref8}. However, the performance of MPC is highly contingent upon the accuracy of the model. When model information is completely unknown, learning-based predictive control (LPC) provides a viable alternative, which employs reinforcement learning (RL) to learn optimal policies. 

RL has emerged as a powerful framework for solving complex decision problems \cite{ref9,ref10}. By approaching learning tasks as interactions between agents and their environment, RL enables agents to learn optimal policies through trial and error.  Literature \cite{ref11,ref12} has researched model-based optimal control algorithms for RL. Literature \cite{ref13,ref14} has investigated the model-free RL optimal control problem. In general, an MPC controller relies on an optimizer to optimize its control policy, whereas LPC uses the RL algorithm as its optimization solver. Literature \cite{ref15} combines the policy gradient algorithm to implement LPC. In literature \cite{ref16,ref17}, Q-learning is combined with MPC control to achieve optimal control of linear systems. In literature \cite{ref18}, the RL-MPC algorithm based on policy iteration is presented. However, RL, as a data-driven algorithm, imposes a heavy computational burden, making it challenging for LPC to achieve real-time execution. Therefore, improving the computational efficiency of RL is a key research direction.

RL is generally combined with optimization methods. The effectiveness and efficiency of optimization methods significantly influence the application of RL algorithms \cite{ref19}. Traditional RL algorithms have relied on gradient descent algorithms. Reference \cite{ref20,ref21} describes the gradient descent based value iteration algorithm and policy iteration, respectively. The parameter optimization of the policy descent algorithms is presented in literature \cite{ref22}. However, the gradient descent algorithms suffer from slow convergence. Currently, RL also applies other optimization methods inspired by the Newton method, as well as its variants. A second order value iteration algorithm is proposed in literature \cite{ref23}. In literature \cite{ref24}, the approximate Newton method is used for policy search. These algorithms converge faster but require computing or approximating Hessian matrices. However, when Hessians are singular or instability occurs, convergence issues arise. Moreover, tuning hyper-parameters for these algorithms remains challenging. Therefore, there is significant research potential in finding a new class of algorithms that balances algorithm stability and convergence speed. 

In this paper, we propose an online LPC for nonlinear systems lacking explicit system dynamics. Instead of obtaining optimal predictive control sequences based on model information as in traditional MPC, the proposed method is based on a data-driven approach using RL to directly obtain optimal policy as the output of the online LPC controller. To facilitate updating in real time, the proposed algorithm incorporates optimal control problem (OCP) method to reduce the computational burden. 

Contributions to this paper are as follows:

 \begin{enumerate}
 \item We combine RL with MPC approach to form a data-driven LPC. In each horizon, the RL algorithm evaluates the value function, generates the optimal policy, and use the policy as the output of the LPC controller.
 \item The algorithm we proposed has super-linear convergence properties with less total number of iterations, and does not require the repeated recalculation of the Hessian matrix during the iterative process, consequently lowers the computational demands, making it better suited for real-time computation.
 \item Our proposed algorithm is stable and still works when the Hessian matrix is a singular matrix.

\end{enumerate}

The structure of the paper is organized as follows: The next section formulates the predictive control problem and provides the background of RL and OCP method. In Section \uppercase\expandafter{\romannumeral3}, we present an RL-based predictive control design that incorporates the OCP method. Subsequently, the convergence of the proposed algorithm is analyzed in Section \uppercase\expandafter{\romannumeral4}. Section \uppercase\expandafter{\romannumeral5} presents a simulation example to verify the effectiveness of the algorithm. Finally, Section \uppercase\expandafter{\romannumeral6} provides a summary of the paper.

\section{Background and Preliminaries}

\subsection{Formulation of Predictive Control Problem}
Consider the following non-linear system
\begin{align*}\label{1}\tag{1}
    x_{k+1}=f(x_k,u_k)
\end{align*}
where $x_k\in\mathbb R^{n}$ and $u_k\in\mathbb{R}^m$ denote the system state and the input, respectively, $k\ge 0$.  $f(x_k,u_k)\in \mathbb {R}^{n}$ is the dynamics of the system. Assume that $f(0,0)=0$ and $f(x_k,u_k)$ is Lipschitz continuous on a compact set $\Omega$ which contains the origin. 

 Consider the following generalized finite horizon cost functional:
 \begin{align*}
    \sum_{i=k}^{k+N_p-1}U(x_i,u_i)+P(x_{k+N_p})\tag{2}\label{2}
\end{align*}
where $U(.)$ denotes the cost of each predictive time step, which is related only to the state and inputs. $P(.)$ denotes the terminal cost. $N_p$ is the predictive horizon.

The objective of predictive control of system \eqref{1} is to find a optimal predictive control sequence $\textbf{u}_k^*$ minimise the performance function \eqref{2} by solving the minimization problem
\begin{align*}
    \min_{\textbf{u}_k} J(x_k,\textbf{u}_k)=& \sum_{i=0}^{N_p-1} U(x_{i|k},u_{i|k})+P(x_{N_p|k})\\
    subject \ to \ & x_{0|k}=x_k \\& x_{i|k+1}=f(x_{i|k},u_{i|k})\\& for \ i=0,1\cdots N_p-1\tag{3}\label{3}
\end{align*}
where $x_{i|k}$ represents the estimated state that simulated $i$ steps ahead of the current state $x_k$, and $\textbf{u}_k=\{u_{0|k},u_{1|k}\cdots u_{N_p-1|k}\}$ is the predictive control sequence.

In model predictive control systems, optimal predictive control sequences can be solved by various types of solvers. However, when the dynamics $f(x_k,u_k)\in \mathbb {R}^{n}$ is unknown, relying solely on the solver cannot obtain the optimal predictive control sequence.

\subsection{Reinforcement Learning}

RL is generally applied in Markov Decision Process (MDP). The stochastic state transition dynamics could described by
\begin{align*}
    \mathbb {P}_{}\left[s_+\,|\, s,a\right] \tag{4}
\end{align*}
where $s$ , $a$ represent the current state-input pair and $s^+$ is the subsequent one. 

The goal of the MDP formulation is to find a mapping from state to action that minimizes ( or maximizes ) the total reward acquired from interacting in an environment for some fixed amount of time
\begin{align*}
   \min G_t=\min\sum^N_{i=t}r(s_i,a_i)   \tag{5}\label{5}
\end{align*}
where $G_t$ is the return represents the total reward , and $r$ is a function that maps the state and inputs to a scalar reward. $N$ is the number of steps.

The return $G_t$ obtained by taking an action $a$ in a given state $s$ can be estimated and the action-value function (Q function) is built
\begin{align*}
    {\cal{Q}}(s,a)=E^\pi(\sum^{N}_{t=0}r(s_t,a_t)|s_0=s,a_0=a)
\end{align*}
where $E^\pi[.]$ denotes the expectation value with respect to the policy $\pi$ and the state transition probability.

Based on Bellman's principle, the Q function $\cal{Q}^\pi$ can be modeled, and the Bellman equation can be given
\begin{align*}
    {\cal{Q}}^\pi(s_t,a_t)&=E^{\pi} \{r(s_t,a_t)+{\cal{Q}}^\pi(s_{t+1},a_{t+1})\}.\tag{6}\label{6}
\end{align*}

The Q-function under the adoption of an optimal policy can be viewed as the optimal Q-function ${\cal Q}^*= {\cal{Q}}^{\pi^*}$. The corresponding optimal policy $\pi^*$ minimizing the total reward can be defined as 
\begin{align*}
    \pi^*(s)&=arg\min_{\pi} E^{\pi^*}(\sum^{N}_{t=0}r(s_t,a_t))\\&=\arg\min_{a}\{{\cal{Q}}^*(s_{t},a_{t})\}.\tag{7}\label{7}
\end{align*}

It can be noticed that the optimal policy \eqref{7} and the optimal representation function \eqref{6} are coupled to each other. RL iterates on Q-functions and policies in the accumulation of data from agent-environment interactions, and ultimately obtains optimal policies.

\subsection{OCP Method}
OCP method is a novel optimization method \cite{ref25}. The OCP method transforms the optimization problem being transformed into an optimal control problem, where the iterative update is designed to minimize the sum of costs at future time instants, thus theoretically giving rise to the optimal algorithm.

Consider a twice differentiable function $L(z)$:  $\mathbb R^d\xrightarrow{} \mathbb R^1$. Our optimization objective is to find the minimum of $L(z)$, i.e., we can rewrite the optimization problem $\min L(z)$ as
\begin{align*}
&\min_v \sum_{i=0}^{M} \left[ L(z_i) + \frac{1}{2} v_i^\top R_d v_i \right] + L(z_{M+1}),
\\&\text{subject to } z_{i+1} = z_i + v_i\tag{8}\label{8}
\end{align*}
where $z_i\in \mathbb{R}^d$ and $v_i\in \mathbb{R}^d$ are the state and control of system \eqref{8}, respectively, $M\ge 0$ is the control time horizon, positive definite matrix $R_d$ is the convergence matrix.

 Based on the reference \cite{ref25}, the following iterative formula can be obtained
 \begin{align*}
z_{i+1} &= z_i - \hat{g}_i(z_i), \quad i = 0, \dots, M, \\
\hat{g}_i(z_i) &= \left( R_d + L''(z_i) \right)^{-1} \left( L'(z_i) + R_d \hat{g}_{i+1}(z_i) \right), \\
\hat{g}_0(z_i) &= \left( R_d + L''(z_i) \right)^{-1} L'(z_i).\tag{9}\label{9}
\end{align*}

Based on this OCP algorithm the parameters can be iterated to obtain the optimal parameters. In Section \uppercase\expandafter{\romannumeral3}, we consider the solution of the nonlinear predictive control using the RL algorithm combined with the OCP method.

The convergence rate of the OCP method is determined by the control matrix $R_d$. As the matrix $R_d$ increases the convergence rate decreases.

\begin{remark}
    The difference between this method and Newton method is that this algorithm can continue to run when the Hessian matrix is a singular matrix. This is because the method introduces a positive definite matrix $R_d$.
\end{remark}

\section{Online LPC Controller Design }
This section describes a methodology for designing online LPC without models in order to achieve predictive control of nonlinear systems. Also, the solver for this algorithm is designed to solve it quickly with a small computational load while online.

\subsection{The Formation of LPC}

In this section, we formulate the LPC problem to be solved at each time step.  The goal of LPC is the same as that of predictive control \eqref{3}, which can be viewed as doing optimal control within a finite horizon. Using dynamic programming principles, the problem can be rewritten as
\begin{align*}
    \min J(x_{j|k})\tag{10}\label{10}
\end{align*}
when $j=0,1\cdots N_{p}-1$,
\begin{align*}
       J(x_{j|k})=&  \sum^{N_p-1}_{i=j} U(x_{i|k},u_{i|k})+P(x_{N_p|k})
\end{align*}
and $j=N_p$,
\begin{align*}
       J(x_{N_p|k})=P(x_{N_p|k})
\end{align*}
where $x_{j+1|k}=f(x_{j|k},u_{j|k})$ and $x_{0|k}=x_k$.

If the problem has a solution, then based on Bellman's principle, the Bellman equation can be obtained:

\begin{align*}
&J^*(x_{j|k})\\ =& 
\begin{cases}
\displaystyle
  \min_{u_{j|k}} \{U(x_{j|k},u_{j|k}) + J^*(x_{j+1|k})\}  & \text{if } j < N_p, \\
P(x_{N_p|k}) & \text{if } j = N_p.
\end{cases}\tag{11}\label{11}
\end{align*}

The optimal predictive control input $ u^*_{j|k}$ in optimal predictive control sequence $\textbf{u}_k^*$ can be obtained based on Bellman's equation,
\begin{align*}
    u^*_{j|k}=arg\min_{u_{j|k}}  \{U(x_{j|k},u_{j|k}) + J^*(x_{j+1|k})\}.\tag{12}\label{12}
\end{align*}

The predictive controller uses only the first element of the predictive control sequence $\textbf{u}_k$ as the output of the LPC controller at each time step
\begin{align*}
    u^*_{k}=u^*_{0|k}.\tag{13}\label{13}
\end{align*}

Since the system dynamics are not known, the Bellman equation \eqref{11} for this nonlinear system is difficult to solve. the problem is solved below using RL instead of a general solver.

\subsection{RL-Based LPC Solver }

In the framework of RL, the original problem \eqref{10} can be viewed as an MDP. At this point, it can be argued that the instantaneous cost $U(x_{j|k},u_{j|k})$ and terminal costs $P(x_{N_p|k})$ in problem \eqref{10} is equivalent to reward $r$ in RL framework \eqref{5}. 

First, consider building a Q-function that responds to the value function under the current state and inputs \cite{ref26}

\begin{align*}
&{\cal Q}(x_{j|k},u_{j|k})\\ =& 
\begin{cases}
\displaystyle
  U(x_{j|k},u_{j|k}) + {\cal Q}(x_{j+1|k},u_{j+1|k})  & \text{if } j < N_p, \\
P(x_{N_p|k}) & \text{if } j = N_p.
\end{cases}\\
    subject \ to \ & x_{0|k}=x_k \\& x_{j+1|k}=f(x_{j|k},u_{j|k})\\& for \ j=0,1\cdots N_p-1.\tag{14}\label{14}
\end{align*}

Then based on the Bellman equation \eqref{11}, the optimal Q-function can be obtained
\begin{align*}
&{\cal{Q}}^*(x_{j|k},u_{j|k})\\ =& 
\begin{cases}
\displaystyle
 U(x_{j|k},u_{j|k}) +\min_{u_{j+1|k}} {\cal{Q}}^*(x_{j+1|k},u_{j+1|k})  & \text{if } j < N_p, \\
P(x_{N_p|k}) & \text{if } j = N_p.
\end{cases}\tag{15}\label{15}
\end{align*}

Based on equation \eqref{15}, the optimal predictive control sequence $\textbf{u}_k^*=\{u^*_{0|k},u^*_{1|k}\cdots u^*_{N_p|k}\}$ is given by
\begin{align*}
    u^*_{j|k}=arg\min_{  u^*_{j|k}} {\cal Q}^*(x_{j|k},u_{j|k}).\tag{16}\label{16}
\end{align*}

In order to obtain the optimal Q-function, the RL algorithm adopts an iterative approach to optimize the Q-function, defining the initial Q-function as ${\cal Q}^{N_p}(x_{N_p|k},u_{N_p|k})=P(x_{N_p|k})$.  And for $j=N_p-1,\cdots1,0$, its update algorithm as follows
\begin{align*}  \label{17} 
&{\cal{Q}}^{j}(x_{j|k},u_{j|k})\\&= U(x_{j|k},u_{j|k})+\min_{ {u}_{j+1|k}}\{{\cal{ Q}}^{j+1}(x_{j+1|k},u_{j+1|k})\} \tag{17}.
\end{align*}

The policy updates coupled to it are as follows
\begin{align*}  \label{18}
u_{j|k}=&\arg\min_{u_{j|k}}\{{\cal{Q}}^{j}(x_{j|k},u_{j|k})\}. \tag{18}
\end{align*}

Since the solution of the policy $u_{j|k}$ is only related to the Q-function, it can be written in the form of state feedback $\pi_{j|k}(x_{j|k})=u_{j|k}$.

Then the output of the LPC controller $u_k$ can be expressed as
\begin{align*}
     u_{k}=\pi_{0|k}(x_{k}).\tag{19}\label{19}
\end{align*}
\begin{figure}
    \centering
    \includegraphics[width=1\linewidth]{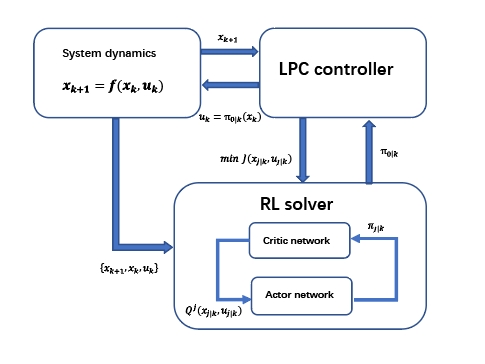}
    \caption{Online LPC Structure}
    \label{Online LPC Structure}
\end{figure}

\subsection{OCP Method Based Online RL Process }

In general, implementing predictive control for nonlinear systems requires a large number of computations. Many studies have taken an offline computing approach. In order to realize its online operation necessarily requires more efficient operation.

For nonlinear systems the Q-function needs to be represented in a parameterized form \cite{ref27}. Therefore, we take the form of a neural network to fit the Q function and the policy respectively.

\subsubsection{Design of the Critic Networks}

According to \eqref{15},\eqref{17}, we define ${\cal{Q}}^{N_p}(x_{N_p|k},u_{N_p|k})=P(x_{N_p|k})$. We fit the each the Q function using the critic network to give the following form:
for $j=N_p-1\cdots ,1,0$,
\begin{align*}
    {\cal{Q}}^j(x_{j|k},u_{j|k})&=W^j_{c}\phi(x_{j|k},u_{j|k})\tag{20}\label{20}
\end{align*}
where $W_{c}^j$ refers to the critic network's weight vector. And $\phi(.)$ represents the vector of activation function. 

Based on Eq. \eqref{17}, for the iterative error of the critic network for each update can be defined as:
for $j=N_p-1$,
\begin{align*}
    e^{N_p-1}_{c}=&U(x_{N_p-1|k},u_{N_p-1|k})+P(x_{N_p|k})\\&-{\cal{Q}}^{N_p-1}(x_{N_p-1|k},u_{N_p-1|k}).
\end{align*}

And  for $j=N_p-2,N_p-3,
\cdots0$,
\begin{align*}
    e^{j}_{c}=&U(x_{j|k},u_{j|k})+{\cal{Q}}^{j+1}(x_{j+1|k},u_{j+1|k})\\&-{\cal{Q}}^{j}(x_{j|k},u_{j|k})\tag{21}\label{21}.
\end{align*}

To minimize the iteration error, we set the loss function for each critic network to be $\xi_{c}^j={e^j_{c}}^2$. Then the weight renewal process is as follows:
\begin{align*}
   W_{c}^{j} &=arg\min_{W_{c}^{j}}\xi_{c}^j
   \\&=arg\min_{W_{c}^{j}} [W^j_{c}\phi(x_{j|k},u_{j|k})-\hat{{\cal Q}}^{j}(x_{j|k},u_{j|k})]^2\tag{22}\label{22}
\end{align*}
where $\hat{{\cal Q}}^{j}(x_{j|k},u_{j|k})=U(x_{j|k},u_{j|k})+{\cal{ Q}}^{j+1}(x_{j+1|k},{{\pi}}^{j+1}(x_{j+1|k}))$.

Equation \eqref{21} can be viewed as an optimization problem trying to find $W^j_c$ minimizing the loss function $\xi^j_c(W^j_c)$. Then the problem $\min \xi^j_c$ can be solved by the OCP method. Then the iterative update of $W^j_c$ can be rewritten in the following form
 \begin{align*}
&W_{c(i+1)}^j = W^j_{c(i)} - \hat{g}_i(W^j_{c(i)}), \quad i = 0, \dots, M, \\
&\hat{g}_i(W^j_{c(i)}) = \alpha_c {\xi^j_c}'(W^j_{c(i)})+\beta_c \hat{g}_{i-1}(W^j_{c(i)}), \\
&\hat{g}_0(W^j_{c(i)}) = \left( R_{dc} + {\xi^j_c}''(W^j_{c(i)}) \right)^{-1} {\xi^j_c}'(W^j_{c(i)})\tag{23}\label{23}
\end{align*}
where $i$ is the number of OCP method iterations. Matrix $R_{dc}$ is positive definite regulating the rate of convergence, $\alpha_c=\left( R_{dc} + {\xi^j_c}''(W^j_{c(i)}) \right)^{-1}$, and $\beta_c=\alpha_c R_{dc}$.

\subsubsection{Design of the Actor Networks}
Similarly, we can design actor networks as
\begin{align*}
    {{\pi}}_{j|k}(x_{j|k})&=W^j_{a}\theta(x_{j|k})\tag{24}\label{24}
\end{align*}
where $j=N_p-1,N_p-2,\cdots0$. $W^j_{a}$ refers to the actor network's weight vector. And $\theta(.)$ represents the vector of activation function.

The iteration error of actor networks $e_{a}^j$ can be set to
\begin{align*}
    e^j_{a}={\cal{Q}}^j_{a}(x_{j|k},\pi_{j|k}(x_{j|k})).\tag{25}\label{25}
\end{align*}

Then, the loss function of actor networks $\xi_{a}^j$ can be defined as $\xi^j_{a}={e^j_{a}}^2$. 

The weights of actor network satisfy
\begin{align*}
   W_{a}^{j} &=arg\min_{W_{a}^{j}}\xi_a^j(W^j_a)
   \\&=arg\min_{W_{a}^{j}} [W^j_{c}\phi(x_{j|k},W^j_a\theta(x_{j|k}))]^2\tag{26}\label{26}.
\end{align*}

The problem is also written as an optimization problem $\min \xi^j_a(W^j_a)$, then using the OCP method, the algorithm for updating the actor network weights is as follows
 \begin{align*}
&W_{a(i+1)}^j = W_{c(i)}^j - \hat{g}_i(W_{a(i)}^j), \quad i = 0, \dots, M, \\
&\hat{g}_i(W_{a(i)}^j) = \alpha_a {\xi^j_a}'+\beta_a \hat{g}_{i-1}(W^j_{a(i)}), \\
&\hat{g}_0(W^j_{a(i)}) = \left( R_{da} + {\xi^j_c}''(W^j_{c(i)}) \right)^{-1} {\xi^j_c}'(W^j_{a(i)})\tag{27}\label{27}
\end{align*}
where $R_{da}$ is a positive definite matrix, $\alpha_a=\left( R_{da} + {\xi^j_a}''(W^j_{a(i)}) \right)^{-1}$, and $\beta_a=\alpha_a R_{da}$.

Based on the above algorithm the optimal weights can be obtained quickly. And, based on remark 1, there is no need for the hessian matrix in the iteration to be guaranteed to be a non-singular matrix. This ensures the stability of the algorithm.

\begin{remark}
Although the OCP algorithm sets an upper iteration limit, it actually does not need to execute M iteration steps during the execution of the algorithm, and the algorithm stops early when the error is smaller than the tolerance error.
\end{remark}

\begin{remark}
    It can be seen that the algorithm runs without iterating the Hessian matrix, repeatedly. This greatly reduces the arithmetic burden. 
\end{remark}

Since the dynamics of the system \eqref{1} are unknown, the predictive control sequence cannot be obtained. But the current predicted control can be obtained based on the output of the actor network
\begin{align*}
    u_k=\pi_{0|k}(x_k).
\end{align*}

 \subsection{Algorithmic Process of Online LPC }

The procedure of the online LPC algorithm is as follows
\begin{breakablealgorithm}

	\caption{  Online LPC algorithm}
    \begin{algorithmic}[1] 

        \STATE Initialize the maximum number of time step $N$, the predictive horizon $N_p$ and the maximum OCP iterations number $M$;
       \STATE  Set the convergence matrices $R_{dc}$, $R_{da}$, and the allowed errors $\gamma$;
   
    \STATE Introduce data pool $\cal S$, containing data points  $\{x^s_l,u^s_l,x^s_{l+1}\},s\in{\{1,2\cdots |{\cal S}|\}}$ with the maximum data number in pool $\eta$;

    \FORALL{$k=0,1,\cdots N $}

  \STATE Initialize the weight vectors $W^{N_p-1}_c$, $\textit{W}_a^{N_p-1}$, respectively;
\FORALL{$j=N_p-1,N_p-2\dots 0$}   
    
      \FORALL{$i=0,1\dots M$}     
      \STATE Calculate the  $e^j_c$ of critic network ${\cal Q}^j(x,u)$ by \eqref{21} using the data in pool $\cal{S}$;

      \STATE Update the critic network weights by \eqref{23};
\IF{$||W^j_{c(i)}-W^j_{c(i-1)}||\le \gamma$ }
\STATE $W^j_{c}=W^j_{c(i)}$;
\ENDIF
\ENDFOR

 \FORALL{$i=0,1\dots M$} 
      \STATE Calculate the iteration error $\xi^j_a$ of actor network $\pi_{j|k}(x)$ by \eqref{25} using the data in pool $\cal S$;

    \STATE Update the actor network weights by \eqref{27};
    \IF{$||W^j_{a(i)}-W^j_{a(i-1)}||\le \gamma$ }
    \STATE $W^j_{a}=W^j_{a(i)}$;
    \ENDIF
    \ENDFOR

    \ENDFOR

    \STATE Execute the output of LPC controller $u_k=\pi_{0|k}(x_k)$ on agent;
\STATE Observe the data point$\{x_k,u_k,x_{k+1} \}$;
\STATE Add the data point $\{x_k,u_k,x_{k+1} \}$ to data pool $\cal S$ and delete the oldest data point;
    \ENDFOR

    \end{algorithmic}

\end{breakablealgorithm}

\begin{remark}
According to Algorithm 1, the algorithm does not require any system dynamics and therefore the algorithm is model-free.
\end{remark}

\begin{remark}
The RL algorithm built into the online LPC algorithm is an off-policy algorithm. Therefore, the policy in the data and the target policy can be different.
\end{remark}
\begin{remark}
The algorithm is an online LPC algorithm, i.e., the agent is trained based on the offline dataset while the real-time data. If the algorithm is trained entirely based on the offline dataset, avoiding the real-time data, Algorithm 1 can be transformed into an offline LPC algorithm.
\end{remark}

\section{Convergence Analysis}\label{sec:Convergence Analysis}
In this section, the convergence of RL algorithms based on the OCP method will be discussed. We will discuss the convergence of the OCP algorithm during the iteration of the single-step RL algorithm.

 \begin{lemma}
     Let $\textbf{z}_k$ be a sequence generated by an iterative method converging to a solution $z_*$. We say that the sequence $\textbf{z}_k$ converges super-linearly to $z_*$ if there exists a positive integer $p\ge1$ such that for sufficiently large $k$, 
     \begin{align*}
         \lim_{k\xrightarrow{} \infty}\frac{||z_{k+1}-z_*||}{||z_k-z_*||^p}=0.
     \end{align*}
 \end{lemma}

\begin{lemma}
    Since $L'(z_k)$ is a factor of every $\hat{g}_i(z_k)$, based on \eqref{9}, it follows that if  $L'(z_k)=0$  is satisfied, then have $\hat{g}_i(z_k)=0$, for $i=0,1\cdots M$.
\end{lemma}

\begin{theorem}
At each iteration step of the update process, the OCP method can be realized to minimize the loss function $\xi^j$ with sufficient data points guaranteed. The weights $W^j_c$ are guaranteed to converge super-linearly to the optimal weight $W^j_{*}$, where the sequence $\{W^j_{(1)},W^j_{(2)}\cdots W^j_{(M)}\}$ is guaranteed to satisfy
\begin{align*}
    |W^j_{(i+1)}&-W^j_{*}|\\&= || (R_d+ {\xi^j}''(W^j_{*}))^{-1}R_d  ||^{i+1}|W^j_{(i)}-W^j_{*}|.
\end{align*}
\end{theorem}

\begin{proof}
    
The updating of the network satisfies equation \eqref{23} and \eqref{27},  the updating process is to minimise the loss function 
\begin{align*}
  \min \xi^{j}(W^j).
   \tag{28}\label{28}
\end{align*}

Based on Lemma 2, assume ${\xi^j}'(W^j_{*}) = 0$, then have $\hat{g}_i(W^j_{*})=0$, for $i=0,1\cdots M$.

Combined with Eq. \eqref{15}, the following relationship can be obtained
\begin{align*}
    &W^j_{(i+1)}-W^j_{*}
    \\=&W^j_{(i)}-W^j_{(i)}-\hat{g}_i(W^j_{(i)})
    \\=&W^j_{(i)}-W^j_{(i)}-[\hat{g}_i(W^j_{*})+\hat{g}_{i}'({W^j_{*}})(W^j_{(i)}-W^j_{*})\\&+o(|W^j_{(i)}-W^j_{*}|)]
    \\\approx&W^j_{(i)}-W^j_{(i)}-[\hat{g}_{i}'({W^j_{*}})(W^j_{(i)}-W^j_{*}))]
    \\=&(I-\hat{g}_{i}'({W^j_{*}}))(W^j_{(i)}-W^j_{*})
    \\=&((R_d+ {\xi^j}''(W^j_{*}))^{-1}R  )^{i+1}(W^j_{(i)}-W^j_{*})\tag{29}\label{29}
\end{align*}
where $o(|W^j_{(i)}-W^j_{*}|)$ is the higher order approximation term. 

Based on \eqref{29} it is known that $       \frac{||W^j_{(i+1)}-W^j_{*}||}{||W^j_{(i)}-W^j_{*}||}=|| (R_d+ {\xi^j}''(W^j_{*}))^{-1}R_d  ||^{i+1}$.

Since $W^j_{*}$, is a minimal value of the function $  \xi^{j}(W^j)$,  it must satisfy ${\xi^j}''(W^j_{*})\ge 0$, then have $|| (R_d+ {\xi^j}''(W^j_{*}))^{-1}R_d  ||< 1$.

Thus we learn that, for a sufficiently large number of iterations $i$, then have $\lim_{i\rightarrow {\infty}}((R_d+ {E^j}''(W^j_{*}))^{-1}R_d  )^{i+1}=0$. Therefore, hold,
\begin{align*}
   \lim_{i\xrightarrow{} \infty} \frac{||W^j_{(i+1)}-W^j_{*}||}{||W^j_{(i)}-W^j_{*}||}=0.
\end{align*}

It can be seen that our algorithm is consistent with Lemma 1, converging super-linearly, so the value of the number of iterations $i$ does not need to be large to converge to the optimal weights. 

Obviously, the weights $W^j$ are guaranteed to converge to the optimal weight $W^j_{*}$. 

The proof is complete.
\end{proof}

\begin{remark}
Understood from other perspectives, the OCP method is generated based on the optimal control problem, which minimizes the control effort while minimizing the sum of the indicators, so the method guarantee that the weights $ W^j$ will be stabilized at the optimal solution $ W^j_{*}$.
\end{remark}

\section{Experiments}
In this section, we present the experimental results of our proposed online LPC algorithm and compare it with the standard LPC algorithm to solve the linear and nonlinear predictive control problem.  All our experiments are operated on the same computer with a Intel Core i5-12100 processor and Matlab R2020a.

\subsection{Linear System Control} \label{linear system}
Consider a discrete-time linear system
\begin{align*}
    x_{k+1}=Ax_k+Bu_k \tag{30}\label{30}
\end{align*}
where $x_k=[x_{1k},x_{2k} ]^\top$ and $u_k$ is the input. The system matrices are expressed as 
\begin{align*}
    A=\begin{bmatrix}
        0.6 & 2\\ 1.5& 0.85
    \end{bmatrix},B=\begin{bmatrix}
        0\\ 0.5
    \end{bmatrix}. 
\end{align*}

Based on \eqref{2}, the control weights are defined the cost function $U(x,u)=x^\top Qx+u^\top R u$, and the terminal cost $P(x)=x^\top Sx$ with $R=1$ and $S=Q=5I$, where $I$ is the identity matrix with appropriate dimensions. The initial state of the system is chosen as $x_0=[1,-0.5]^\top$. And the predictive horizon is $N_p=10$. 

Based on the LQ algorithm \cite{ref28}, we know that the Riccati matrix $P_{10}$ and the corresponding optimal feedback matrix $K_{10}$ for this system can be expressed as
\begin{align*}
     P_{10}=\begin{bmatrix}
        20.20 & 23.56\\ 23.56& 57.67
    \end{bmatrix},K_{10}=\begin{bmatrix}
        3.26&3.12
    \end{bmatrix}.\tag{31}\label{31}
\end{align*}

Algorithm 1 is considered to be applied to achieve predictive control of the system \eqref{30}. Since the system is linear, the activation function we can choose for the critic network is $\phi(x_k,u_k)=[x_{1k}^2,x_{2k}^2,x_{1k}x_{2k},x_{1k}u_k,x_{2k}u_k,u_k^2]^\top$, and for the actor network is $\theta(x_k)=[x_{1k},x_{2k}]^\top$. The data pool size is $|{\cal{S}}|=30$. The convergence matrices for both networks are $R_{dc}=R_{dc}=0.1I$. The control horizon of OCP method is $M=100$, and the RL iterative tolerance error $\gamma=10^{-5}$ . The initial weight values $W_c^{N_p-1}$ and $W^{N_p-1}_a$ are randomly drawn from $[-1,1]$.

We control the system \eqref{30} based on the above online LPC algorithm, and Figure. \ref{fig:1} shows the state trajectory. Figure. \ref{fig:2} represents the trajectory of the control input.

\begin{figure}
    \centering
    \includegraphics[width=0.7\linewidth]{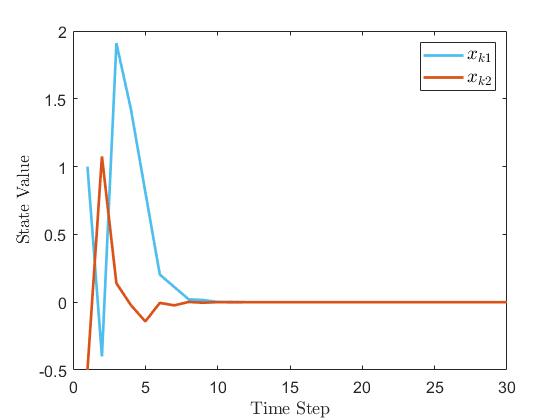}
    \caption{State trajectory for linear system}
    \label{fig:1}
\end{figure}

\begin{figure}
    \centering
    \includegraphics[width=0.7\linewidth]{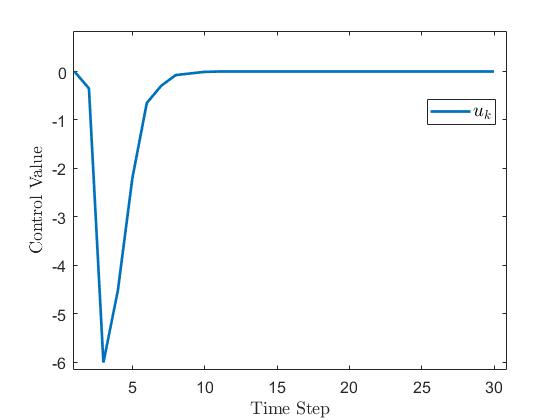}
    \caption{Control input trajectory for linear system}
    \label{fig:2}
\end{figure}

According to Figures. \ref{fig:1} and \ref{fig:2}, it can be seen that the online LPC algorithm can be realized to make the system \eqref{30} stable.

Based on the introduction in Section III, Algorithm 1 has a fast convergence rate. Its convergence speed is analyzed below.
We similarly set up the traditional gradient descent based LPC algorithm for comparison. The networks for both LPC algorithms use the same activation functions as above. The learning rate is set to $\alpha=0.05$. All other settings are the same as in Algorithm 1 above. Figure. \ref{fig:3} compares the convergence speed of the two RL algorithms in each horizon.

\begin{figure}[htbp]
  \centering
  \subfloat[]
  {\includegraphics[width=0.2\textwidth]{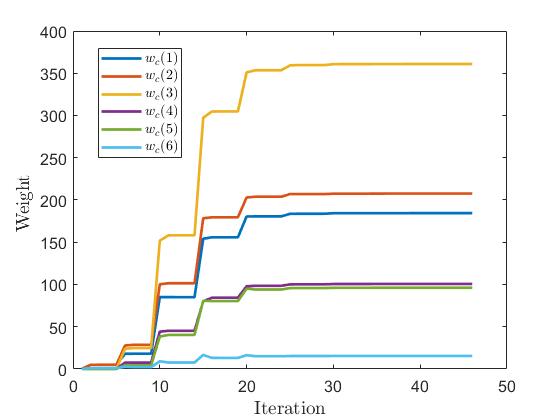}\label{fig:subfig3:1}}
   \     
  \subfloat[]
  {\includegraphics[width=0.2\textwidth]{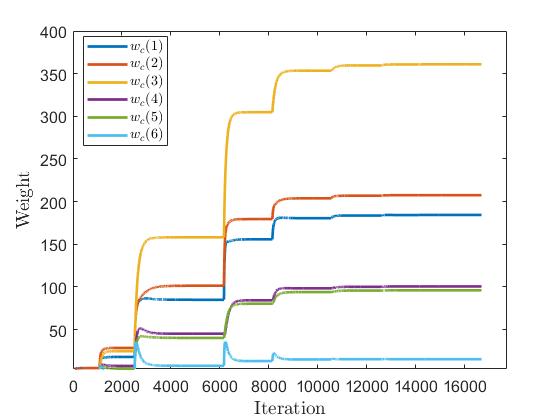}\label{fig:subfig3:2}}
   \
   
  \subfloat[]
  {\includegraphics[width=0.2\textwidth]{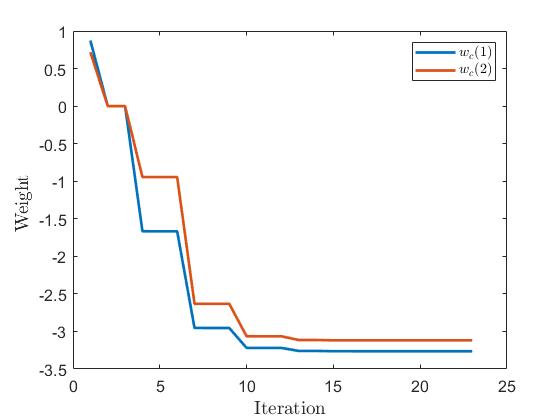}\label{fig:subfig3:3}}
   \
  \subfloat[]
  {\includegraphics[width=0.2\textwidth]{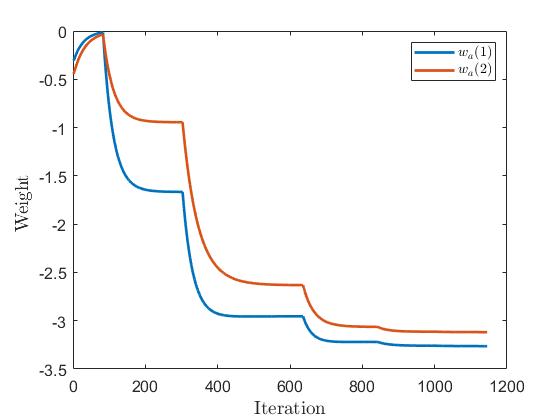}\label{fig:subfig3:4}}
   \
  \caption{Convergence analysis of critic and actors’ weights. (a) Trajectory of critic’s weights with proposed algorithm.  (b) Trajectory of critic’s weights with traditional algorithm.  (c) Trajectory of actor’s weights with proposed algorithm. (d) Trajectory of actor’s weights with traditional algorithm.  }\label{fig:3}
\end{figure}

 Figure. \ref{fig:3} shows that the weights of actor network finally converge to $W_c^0=[-3.264,-3.118]$. And based on Eq. \eqref{31}, the optimal predictive law is $u_k=-K_{10}x_k$. It can be verified that the proposed algorithm finds the optimal policy. The RL algorithm based on the OCP method has a fast convergence rate. Both actor network and critical network, the recommended algorithms converge within 50 iterations. And the gradient based LPC algorithm has much higher number of iterations than the recommended algorithm.

\subsection{Non-linear System Control}
Consider the following nonlinear system
\begin{align*}
  &  x_{k+1}=\\&\begin{bmatrix}
        \sin{(x_{k1})}+0.1x_{k2}+0.1x_{k1}^2\\-1.2x_{k1}+0.8x_{k2}+0.1\sin{(u_k+x_{k1})}+0.2x_{k2}u_k
    \end{bmatrix}\tag{32}\label{32}
\end{align*}
where $x_k=[x_{1k},x_{2k} ]^\top$ and $u_k$ is the input. 

The initial state value $x_k=[0.9 , -0.7]^\top$, the control weights are defined the cost function $U(x,u)=x^\top Qx+u^\top R u$, and the terminal cost $P(x)=x^\top Sx$ with $R=1$ and $S=Q=2I$, where $I$ is the identity matrix with appropriate dimensions. The predictive horizon is $N_p=10$. Compared to linear system \eqref{30}, nonlinear system \eqref{32} is more complex and therefore require more complex activation functions for both networks. The activation function of critic network is re-defined as $\phi(x_k,u_k)= 
[x_{1k}^2,x_{2k}^2,x_{1k}x_{2k},x_{1k}u_k,x_{2k}u_k,u_k^2,x_{1k}^3,x_{2k}^3,x_{1k}^2x_{2k},\\x_{1k}x_{2k}^2, x_{1k}^2u_k,x_{2k}^2u_k,x_{1k}x_{2k}u_k ]$, as well as the actor network's activation function $\theta(x_k)=[x_{1k},x_{2k},x_{1k}^2,x_{2k}^2,x_{1k}x_{2k}]$. The rest of the configuration is the same as in Subsection V-A. 
Figures. \ref{fig:4} and \ref{fig:5} illustrate the change in state trajectories and the change in control input trajectories during the online LPC iteration process. 

\begin{figure}[H]
    \centering
    \includegraphics[width=0.7\linewidth]{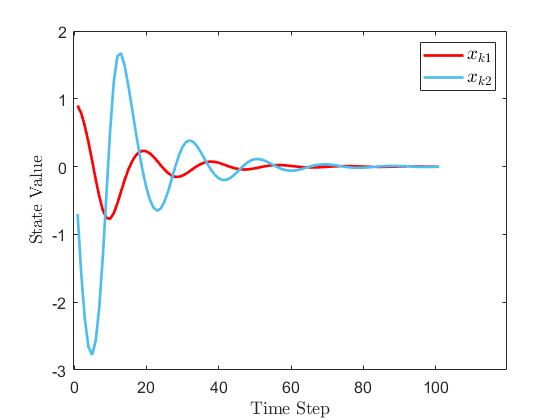}
    \caption{State trajectory for non-linear system}
    \label{fig:4}
\end{figure}

\begin{figure}[H]
    \centering
    \includegraphics[width=0.7\linewidth]{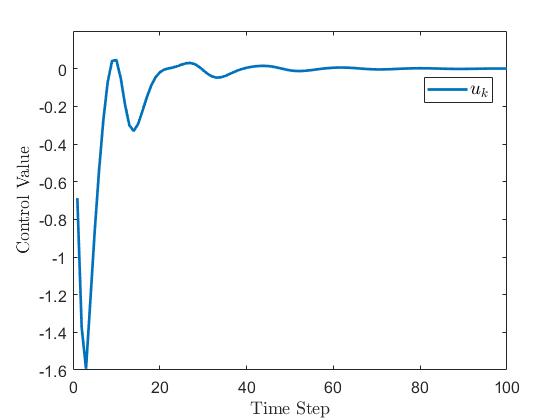}
    \caption{Control input trajectory for non-linear system}
    \label{fig:5}
\end{figure}

It can be seen that the online LPC algorithm can stabilize the nonlinear system \eqref{32}. Then, Figure. \ref{fig:6} illustrates the convergence of the networks weights with in each predictive horizon and compares the number of iterations required for the convergence of the RL algorithms based on the OCP method and gradient descent, respectively.
\begin{figure}[htbp]
  \centering
  \subfloat[]
  {\includegraphics[width=0.2\textwidth]{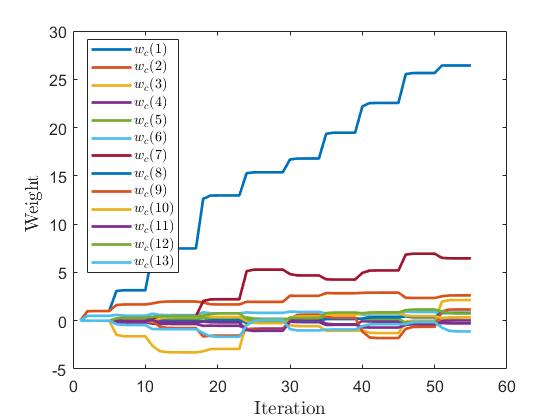}\label{fig:subfig6:1}}
   \     
  \subfloat[]
  {\includegraphics[width=0.2\textwidth]{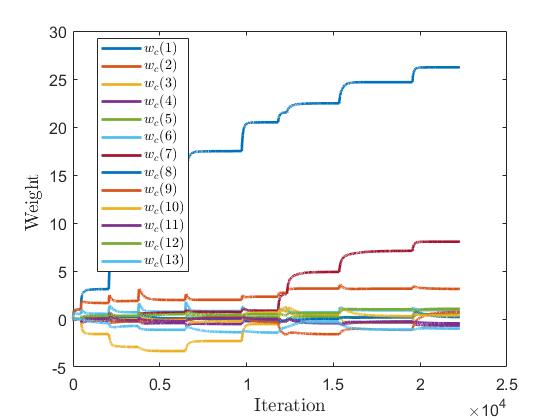}\label{fig:subfig6:2}}
   \
   
  \subfloat[]
  {\includegraphics[width=0.2\textwidth]{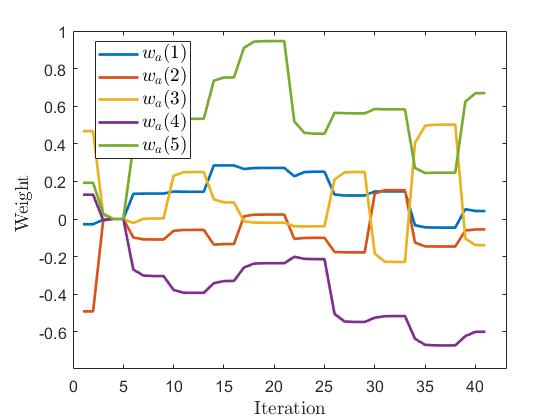}\label{fig:subfig6:3}}
   \
  \subfloat[]
  {\includegraphics[width=0.2\textwidth]{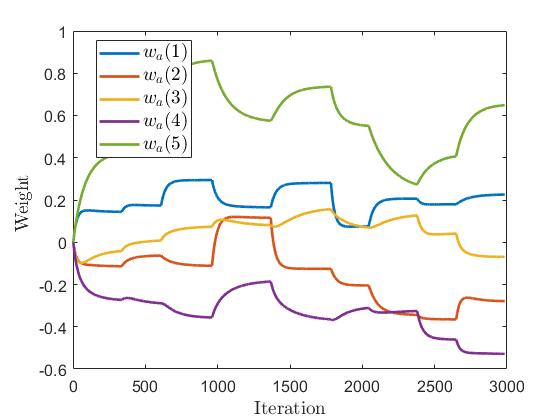}\label{fig:subfig6:4}}
   \
  \caption{Convergence analysis of critic and actors’ weights. (a) Trajectory of critic’s weights with proposed algorithm.  (b) Trajectory of critic’s weights with traditional algorithm.  (c) Trajectory of actor’s weights with proposed algorithm. (d) Trajectory of actor’s weights with traditional algorithm.  }\label{fig:6}
\end{figure}

\begin{table*}[htbp]
    \centering
    \begin{tabular}{|c|c|c|c|} \hline  
         System& LPC solver& Average run time &Iterations\\ \hline  
          \multirow{2}*{linear system \eqref{30}}&  Gradient descent based RL& 0.55s&16032\\   \cline{2-4} ~
 & OCP method based RL& 0.05s&43\\\hline   
         \multirow{2}*{non-linear system \eqref{32}}&  Gradient descent based RL& 1.07s&24150\\ \cline{2-4}
 ~& OCP method based RL& 0.05s&48\\ \hline 
    \end{tabular}
   
    \caption{Comparison average over 20 trials in each horizon }
    \label{tab:tab1}
\end{table*}

Comparing Fig. \ref{fig:subfig6:1} and Fig. \ref{fig:subfig6:2}, the more complex the system is, the slower the algorithm converges. However, it can be seen that the recommended algorithm converges much faster than the traditional LPC algorithm based on gradient descent algorithm. According to the analysis of Fig. \ref{fig:subfig6:1}, after iterations the final networks weights $W_c^0$ and $W^0_a$ finally converges to $W^0_c=[18.784,3.254,2.617,-0.213,0.557,1.483,-6.122,0.260,\\5.013,1.158,0.3624,0.944,0.272]$, and $W_a^0=[0.042,-0.055,-0.135,-0.599,0.670]$. The above experiments prove that the proposed algorithm can find the optimal policy faster and requires fewer iterations, which makes it more suitable for solving online.

Next we compare the computational efficiency of each algorithm. We compute the average computation time and number of iterations for RL in each predictive horizon. The convergence time of each algorithm is recorded for 20 predictive horizons with $N_p=10$ for each of the two systems \eqref{30}, \eqref{32} and the Table 1 is obtained.

It can be seen that the recommended algorithm has less computational load while its computation time is shorter. Online LPC algorithms require real-time computation and require shorter runtime. It can be seen that the gradient descent based LPC algorithm can only operate offline when the runtime is less than $0.06s$. While the recommended algorithm can continue to run in real time.

\subsection{Non-linear Trajectory Tracking Control}

We consider a nonlinear system, namely Van der Pol's oscillator \cite{ref29}, its dynamic equations are as follows
\begin{align*}
    &\dot x_1=x_2\\
    &\dot x_2=(1-x_1^2)x_2-x_1+5u  \tag{33}\label{33}.
\end{align*}

Next, the model \eqref{33} is discretized, when the sampling time is $\Delta t=0.1s$, then the discrete model is as follows
\begin{align*}
  x_{k+1}=\begin{bmatrix}
      x_{k1}+0.1x_{k2}\\
      x_{k2}+0.1(1-x_{k1}^2)x_{k2}-0.1x_{k1}+0.5u_k
  \end{bmatrix}\tag{34}\label{34}.
\end{align*}

We show its reference trajectory $r_k=[r_{k1} \ r_{k2}]^\top$. The reference trajectory dynamics are as follows
\begin{align*}
    r_{k+1}=\begin{bmatrix}
        \sin{(0.1k)}\\\cos(0.1k)
    \end{bmatrix}.
\end{align*}

The initial state value $x_k=[2 , -1]^\top$, and the initial tracking trajectory is $r_k=[1 , 0]^\top$. After that, we define the tracking error $e_k=[x_{k1}-r_{k2} , x_{k2}-r_{k2}]^\top$.

The reward function for tracking control is expressed as $U(e,u)=e^\top Qe+u^\top Ru$, accompanied by a terminal cost defined as $P(e)=e^\top Se$. In this context, the parameters are set to $R=1$ and $S=Q=10I$, where $I$ represents the identity matrix of suitable dimensions. The predictive horizon is established at $N_p=10$. 

For ease of computation, we combine the system dynamics and trajectory dynamics into the augmented system $X_k=[x_{k1} , x_{k2} , r_{k1} , r_{k2}]^\top$. The basis function $\phi(X_k,u_k)$ is chosen to be polynomials of multiple orders in $X_k$ and $u_k$.  And the basis function $\theta(X_k)$ is picked as polynomials of multiple orders in $X_k$.

To validate the tracking performance and tracking accuracy of the algorithm, a set of controlled experiments are shown then. Here PID controller is used for tracking control of this system. Its control law is regulated by trial-and-error to find the following control law
\begin{align*}
    u_k=&-K_{P1}e_{k1}-K_{P2}e_{k2}-K_{I1}\sum^k_{i=0}e_{i1}-K_{I2}\sum^k_{i=0}e_{i2}\\&-K_{D1}(e_{k1}-e_{k-1,2})-K_{D2}(e_{k2}-e_{k-1,2})
\end{align*}
where $K_{P1}=0.9,K_{P2}=0.8$ ,$K_{I1}=K_{I2}=0.5$ and $K_{D1}=K_{D2}=0.01$.

The following figure describes the tracking trajectory of the system \eqref{34} under the two control methods.
\begin{figure}[htbp]
  \centering
  \subfloat[]
  {\includegraphics[width=0.2\textwidth]{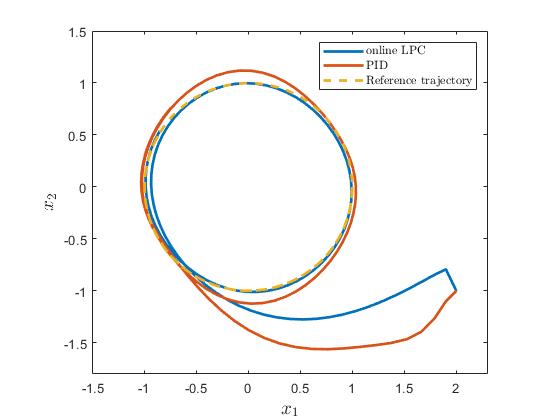}\label{fig:subfig8:1}}
   \     
  \subfloat[]
  {\includegraphics[width=0.2\textwidth]{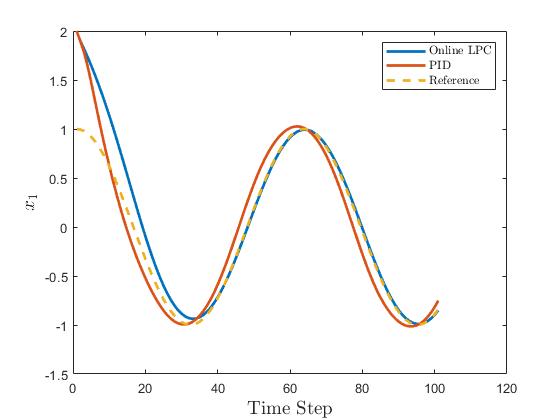}\label{fig:subfig8:2}}
   \
   
  \subfloat[]
  {\includegraphics[width=0.2\textwidth]{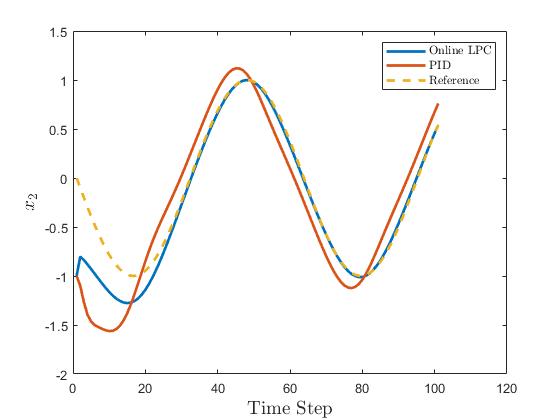}\label{fig:subfig8:3}}
   \
  \subfloat[]
  {\includegraphics[width=0.2\textwidth]{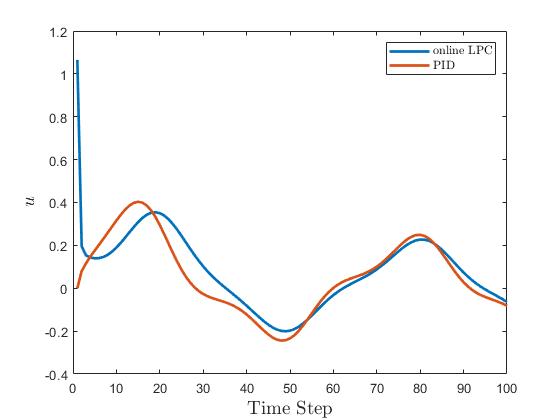}\label{fig:subfig8:4}}
   \
  \caption{Convergence analysis of critic and actors’ weights. (a) Comparison of tracking trajectories.  (b) Comparison of tracking performance of $x_1$ trajectories.  (c) Comparison of tracking performance of $x_2$ trajectories.(d) Comparison of control input of both control method.  }\label{fig:8}
\end{figure}

Figure \ref{fig:8} illustrates the tracking effect of the online LPC algorithm. As shown in Fig. \ref{fig:subfig8:1}, although the PID algorithm has been well-tuned, the recommended algorithm has a smaller tracking error and faster response time compared to the PID algorithm. Figure \ref{fig:subfig8:2} and \ref{fig:subfig8:3} show that our proposed algorithm can track reference trajectory either in $x_1$ or $x_2$ dimensions with tracking errors less than 0.04 and outperforms the PID algorithm in both cases. Figure \ref{fig:subfig8:4} illustrates that the control input of our proposed algorithm is smoother.

After several tests, the operation time of the online LPC algorithm is around 0.06s, which is less than the sampling time $\Delta t$ and cannot be achieved by the gradient descent based LPC algorithm, thus the proposed algorithm can run online.

\section*{Conclusion}
In this paper, we propose an online LPC algorithm to address the predictive control of nonlinear systems without requiring knowledge of system dynamics in real time. To ensure online operation and stability, we combine the OCP method with the RL algorithm as the LPC solver. The proposed algorithm has a low computational load, significantly reducing runtime. Moreover, the algorithm is stable even when the Hessian matrices become singular during the computation. We have demonstrated that the algorithm converges to both the optimal critic and actor networks. Simulations reveal that the proposed algorithm is efficient and suitable for online operation. This paper focuses on solving online LPC problems, and future work will explore LPC with constraints for practical applications.

\ifCLASSOPTIONcaptionsoff
  \newpage
\fi

\bibliographystyle{IEEEtran}
  
\bibliography{ref}

\end{document}